\documentclass{article}

\usepackage{PRIMEarxiv}

\usepackage[utf8]{inputenc} 
\usepackage[T1]{fontenc}    
\usepackage{hyperref}       
\usepackage{natbib}
\usepackage{doi}
\usepackage{url}            
\usepackage{booktabs}       
\usepackage{amsfonts}       
\usepackage{amsmath}
\usepackage{amssymb}
\usepackage{float}
\usepackage{amsthm}
\usepackage{nicefrac}       
\usepackage{microtype}      
\usepackage{lipsum}
\usepackage{fancyhdr}       
\usepackage{graphicx}       

\usepackage{graphicx}
\usepackage{tikz-qtree}
 \usetikzlibrary {positioning} 

\usepackage{enumerate}

\RequirePackage{tikz}
\RequirePackage{tikz-qtree}

\theoremstyle{thmstyleone}%
\newtheorem{theorem}{Theorem}
\newtheorem{proposition}[theorem]{Proposition}%
\newtheorem{lemma}[theorem]{Lemma}

\theoremstyle{thmstyletwo}%
\newtheorem{remark}{Remark}%

\theoremstyle{thmstylethree}%
\newtheorem{definition}{Definition}%

\graphicspath{{media/}}     

\usepackage{bbm}

\usepackage{verbatim,graphics,indentfirst} 
\usepackage{url}
\usepackage{stmaryrd}

\usepackage{latexsym}
\usepackage{enumerate}
\usepackage{graphicx}
\usepackage[T1]{fontenc}
\usepackage{lmodern}

\usepackage[utf8]{inputenc}
\usepackage{graphicx}
\usepackage{amsfonts}
\usepackage{nicefrac}
\usepackage{array}
\usepackage{appendix}

\usepackage{tikz}
\usepackage{tikz-qtree}
\usepackage[]{forest}

\usepackage{xspace}
\usepackage{marginnote}
\usepackage{bussproof}

\usepackage{proof}

\usepackage{pdflscape}

\usepackage{tikz}

\usepackage{mathdots}

\newcommand{\Di}{\Diamond}
\renewcommand{\phi}{\varphi}

\input{macros.tex}

\title{Axiomatizing the Logic of Ordinary Discourse
}

\author{\href{https://orcid.org/0000-0003-3240-386X}{\includegraphics[scale=0.06]{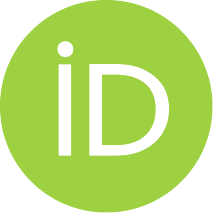}}\hspace{1mm}%
  Vitor Greati \\
  Bernoulli Institute \\
  University of Groningen \\
  Groningen, The Netherlands\\
  \texttt{v.rodrigues.greati@rug.nl} \\
   \And
  \href{https://orcid.org/0000-0002-6941-7555}{\includegraphics[scale=0.06]{orcid.pdf}\hspace{1mm}}Sérgio Marcelino \\
  SQIG – Instituto de Telecomunicações\\
  Departamento de Matemática -- Instituto Superior Técnico \\
  Lisboa, Portugal\\
\texttt{smarcel@math.tecnico.ulisboa.pt} \\
  \And
  \href{https://orcid.org/0000-0003-1364-5003}{\includegraphics[scale=0.06]{orcid.pdf}}\hspace{1mm}%
  Umberto Rivieccio \\
  Departamento de Lógica, Historia y Filosofía de la Ciencia \\
  Universidad Nacional de Educación a Distancia \\
  Madrid, Spain\\
\texttt{umberto@fsof.uned.es} \\
}

\begin{document}
\maketitle

\begin{abstract}
Most non-classical logics are subclassical, that is,
every inference/theorem they validate
is also
valid classically. 
A notable exception is the three-valued propositional \emph{Logic of Ordinary Discourse}
($\CooOL$) proposed and extensively motivated by W.S.~Cooper 
as a more adequate candidate for formalizing everyday  reasoning (in English).
$\CooOL$ challenges 
classical logic not only by rejecting some 
theses, but also by accepting non-classically valid principles,  
such as
so-called Aristotle's and Boethius' theses. 
Formally,
$\CooOL$  shows a number of 
unusual features
--
it is non-structural, connexive, paraconsistent and contradictory
--
making it all the more interesting for the mathematical logician. 
We present 
our recent findings on 
$\CooOL$ and its structural companion (that we call $\SOL$).
We introduce  Hilbert-style multiple-conclusion calculi for
$\CooOL$ and $\SOL$
 that are both modular and analytic,
and easily allow us to obtain single-conclusion axiomatizations. We prove
that 
$\SOL$
is algebraizable and single out its equivalent 
semantics, which turns out to be a discriminator variety generated by a three-element algebra. 
Having observed that $\SOL$ can express the connectives of other three-valued logics,
we prove that it is definitionally 
equivalent to an expansion of the three-valued logic $\mathcal{J}3$ of D'Ottaviano and da Costa, itself an axiomatic extension of paraconsistent 
Nelson logic.
\end{abstract}

\section{Introduction}
\label{sec:intro}
Most non-classical propositional systems
result from weakening  Boolean two-valued  logic one way or another:
 this applies in general to  logics in the fuzzy, many-valued, relevance and substructural families. 
 The majority of three-valued logics are also subclassical in this sense (see~\cite{threevalchapter} for an  overview). 
A remarkable 
exception 
is represented  by the 
\emph{Logic of Ordinary Discourse} introduced by W.S.~Cooper~\cite{Cooper1968}, a propositional system (henceforth denoted $\CooOL$) 
   remarkable in other ways as well.

As the name indicates, $\CooOL$ was proposed as a logic 
for formalizing \emph{ordinary discourse}, that is, it was  
designed to model the  everyday usage of  natural language connectives (in particular   the \emph{if-then}).
 In this respect, the main point of departure of $\CooOL$ from  classical logic concerns those conditional sentences
having a false antecedent. Following a famous suggestion of Quine, Cooper regards such sentences as lacking a truth value,
but formally represents this condition  by employing a third (or ``gap'') value $\uv$ besides the classical $ \tv$ and $ \bv $. 
 $\CooOL$ may thus be defined by means of {valuations over}
a three-element logical matrix (
Figure~\ref{fig:truth-tables2})
 over the set of truth values
$\{ \bv, \tv, \uv \}$ where, interestingly, both $\tv$ and the gap value $\uv$ are designated.
{However, as we shall soon discuss, $\CooOL$ is \emph{not} determined
by this three-valued matrix in the usual sense, for in its definition not all
valuations are allowed.}
\begin{figure}
    \centering
    \begin{tabular}{@{}c|ccc@{}}
        \toprule
         $\land$ & \uv & \tv & \bv \\
        \midrule
         \uv & \uv & \tv & \bv\\
         \tv & \tv & \tv & \bv \\
         \bv & \bv & \bv & \bv \\
        \bottomrule
    \end{tabular}
    \quad
    \begin{tabular}{@{}c|ccc@{}}
        \toprule
         $\lor$ & \uv & \tv & \bv \\
        \midrule
         \uv & \uv & \tv & \bv\\
         \tv & \tv & \tv & \tv \\
         \bv & \bv & \tv & \bv \\
        \bottomrule
    \end{tabular}
    \quad 
    \begin{tabular}{@{}c|ccc@{}}
        \toprule
         $\to$ & \uv & \tv & \bv\\
        \midrule
         \uv & \uv & \tv & \bv\\
         \tv & \uv & \tv & \bv\\
         \bv & \uv & \uv & \uv\\
        \bottomrule
    \end{tabular}
    \quad 
    \begin{tabular}{@{}c|c@{}}
        \toprule
         & $\neg$ \\
         \midrule
         \uv & \uv\\
         \tv & \bv\\
         \bv & \tv\\
         \bottomrule
    \end{tabular}
    
%
\caption{Truth tables for $\CooOL$~\cite[Sec.~5]{Cooper1968}.}
\label{fig:truth-tables2}
\end{figure}



$\CooOL$
validates certain formulas that are not classical tautologies
-- notably 
so-called \emph{Boethius' thesis} 
$
(\phi \to \psi) \to \nnot (\phi \to \nnot \psi)$
and \emph{Aristotle's thesis} 
$ \nnot (\nnot \psi \to  \psi)$ 
--
suggesting that it 
may be viewed as
a  \emph{connexive} logic~\cite{WanConnexive05}. 
On the other hand, some classical tautologies 
 (such as the explosion {principle} $\phi \to( \nnot \phi \to \psi) $)  are not unconditionally valid in $\CooOL$%
, making
it incomparable with (not weaker than) classical propositional logic. 

{
Cooper~\cite{Cooper1968} 
dicusses several examples supporting
the claim that $\CooOL$
provides a better formal model  of ordinary discourse  than classical 
logic. Here
we  may  mention but one, 
a puzzle that can be traced back to Aristotle, as reconstructed by J.~\L ukasiewicz~\cite[pp.~312-3]{Cooper1968}.
In Aristotle's view, no proposition can imply its own negation, hence
$ \nnot (\nnot \psi \to  \psi)$ should be a tautology. 
But, classically,  $(\phi \to \psi) \land (\nnot \phi \to \psi) \vdash \nnot \psi \to \psi $.
Hence, by contraposition,  $(\phi \to \psi) \land (\nnot \phi \to \psi)$ should be contradictory,
which seems counter-intuitive. 
Now, 
in $\CooOL$,
Aristotle's thesis   $\nnot( \nnot \psi \to \psi)$  is indeed a tautology. 
We also have, as in classical logic, that   $(\phi \to \psi) \land (\nnot \phi \to \psi) \vdash \nnot \psi \to \psi $ 
and 
$(\phi \to \psi) \land (\nnot \phi \to \psi) $ is satisfiable. 
$\CooOL$, however, rejects  the contraposition rule
($\alpha \vdash  \beta$ does not entail  $\nnot  \beta \vdash  \nnot \alpha$), so
 $\nnot (\nnot \psi \to \psi) \not \vdash \nnot ((\phi \to \psi) \land (\nnot \phi \to \psi))$, avoiding 
the counter-intuitive consequence.}
We shall not further
discuss  the adequacy of  $\CooOL$ with respect to the proposed applications
(refer to~\cite{Cooper1968}), but we wish to draw attention
to some of its unusual  features.  

For one thing, $\CooOL$ is not even a logic in the Tarskian sense, for its consequence relation is 
 \emph{non-structural}, i.e.~not closed under uniform substitution.
This  is witnessed, for instance, by the explosion principle, which  $\CooOL$ endorses for atomic formulas
($p, \nnot p \vdash q $) but not for arbitrary ones ($\phi, \nnot \phi \not \vdash \psi $). 
In consequence, $\CooOL$ is not characterized by truth tables in the usual way;  it can, however,  be semantically characterized 
by the three-valued matrix described in Figure~\ref{fig:truth-tables2} if we require valuations to assign only classical values 
($\bv$ or $\tv$)
to the propositional variables.

Let us mention a few    other unusual 
features of $\CooOL$, 
whose 
language 
consists of a conjunction $(\land)$, a disjunction $(\lor)$, an implication $(\to)$ and a negation $(\nnot)$.
The usual De Morgan laws relating conjunction and disjunction via negation hold, and indeed either $\land$
or $\lor$ could be omitted from the primitive signature. 
However, neither the distributive nor the absorption laws between conjunction and disjunction hold. 
{Indeed, 
$\land$ and $\lor $
do not give rise to a lattice structure in the expected way (cf.~Proposition~\ref{prop:ol3asn4} below); for this reason they have been called
\emph{quasi-conjunction} and \emph{quasi-disjunction} in the literature (see~\cite[Sec.~6.3]{egre20211} for a discussion of
their motivation). A novel contribution to an informal reading of these connectives may perhaps be based on our
algebraic analysis of $\CooOL$. As we shall see (Section~\ref{sec:alg}), $\CooOL$ may be viewed as
an expansion of
Da Costa and D'Ottaviano's 
logic $\mathcal{J}3$, which is in turn 
an axiomatic extension of paraconsistent 
Nelson logic. These are substructural logics based on residuated structures having both 
a lattice (or \emph{additive}) conjunction and a \emph{multiplicative} conjunction; 
on the 
algebraic models, the latter is realized by a monoid operation  having (in our notation)
$\uv$ as neutral element. This suggests that 
Cooper's connective $\land$ is perhaps best thought of as
a multiplicative conjunction having $\lor$ as its De Morgan dual (defined by 
$x \lor y : = \nnot (\nnot x \land \nnot y)$), the truncated sum of \L ukasiewicz logic being another example of this kind of disjunction.
}

The behaviour of $\lor $ 
in $\CooOL$ is indeed peculiar, even 
in isolation. 
In fact, the preceding observations entail that  disjunction introduction is not a sound rule (in general, we have
$\phi \not \vdash \phi \lor \psi$).
In many logics, another key feature of the disjunction is that the truth value of a formula
$ \phi \lor \psi $ is designated if and only if either  $\phi$ or $\psi$ is assigned a designated value. This does not hold
in $\CooOL$ (one has, for instance $\uv \lor \bv = \bv$). In consequence, a classical tautology such as $\phi \lor (\phi \to \psi)$
is
contigent 
in $\CooOL$, although 
every valuation assigns a designated value either to $\phi$ or to $\phi \to \psi$.
%
%
The behaviour of $\to$ 
within $\CooOL$ appears to be more standard, at least in isolation.
Unlike the disjunction, the implication is {not} definable from the other connectives in any of the  usual ways.
Indeed, it is easy to see
that 
$\to$ is not definable at all in the language  $\{ \land, \lor, \nnot \}$, for
the set $\{ \bv, \tv \}$ is closed under all these operations but not under the implication.
On the other hand, the truth constants (all three of them)
are definable. For instance, one may let  $\uv := \phi \to (\nnot \phi \to \phi )$,
$\tv := p \lor \neg p$ and 
$\bv := \nnot \tv$ 
(here $p$ needs to be atomic, while 
 $\phi$ is arbitrary). 

Lastly, observe that  $\CooOL$ is not only a paraconsistent (for $\phi, \nnot \phi \not \vdash \psi $) but actually
a \emph{contradictory} logic in H.~Wansing's sense~\cite{WansingForthcoming-WANCLI-2},
that is, there exists a formula (e.g.~$\uv$, defined as above)
such that both $\vdash \uv$ and $\vdash \nnot \uv$ hold.  This 
has the interesting consequence that $\CooOL$
does not admit any non-trivial structural extension: if we were to close the consequence relation of $\CooOL$ under uniform substitutions, then any formula $\phi$ would be valid,
 being (by the explosion rule, now applicable to arbitrary formulas) a consequence of the set of valid formulas $\{ \uv, \nnot \uv \}$.

We introduce
 multiple- and single-conclusion Hilbert-style axiomatizations 
for $\CooOL$ and for its structural companion $\SOL$ (Sections~\ref{sec:axiomatization}),
{ thus filling the gap in the literature concerning a standard axiomatization
of $\CooOL$}.
Our calculi are 
  modular {(i.e.,~obtained  by joining independent calculi over smaller signatures)}, easily allowing us to characterize a number of fragments of $\CooOL$/$\SOL$. 
  In the multiple-conclusion setting,
  they are also analytic.
  {We obtain these axiomatizations 
  via the methods developed in \cite{ss1978,marcelino19woll}}.
In Section~\ref{sec:alg} we give an alternative axiomatization for $\SOL$
and we
prove that $\SOL$  is algebraizable
with respect to the quasi-variety $\mathbb{OL}$ (coinciding with the variety) generated
by the 
algebra $\OL_3$, the 
algebraic reduct 
of Cooper's three-valued
matrix. 
$\mathbb{OL}$ 
turns out to be a discriminator variety
(Theorem~\ref{the:quasi-var-gen}), making $\SOL$ a nearly functionally complete logic
(adding to its language either the constant $\tv$ or $\bv$ makes it functionally complete).
Indeed, we show that $\SOL$ is definitionally equivalent to an expansion of
Da Costa and D'Ottaviano's three-valued logic $\mathcal{J}3$, which is in turn 
an axiomatic extension of paraconsistent 
Nelson logic.
The final Section~\ref{sec:fut} discusses potential future developments,
notably
a more extensive study of the algebraic counterpart of $\SOL$ (especially in connection
with 
the other above-mentioned non-classical logics), and 
the extension of the present approach to other logics definable from the truth tables of 
$\CooOL$ (e.g., 
the algebraizable fragments of $\SOL$,  logics
resulting from alternative choices of designated elements or those determined by
the definable orders of $\OL_3$).


\section{Logical preliminaries}
\label{sec:prelims}
A \emph{propositional signature} $\Sigma$ is a collection of symbols called \emph{connectives}; to each of them is assigned a natural number called
\emph{arity}. Given a countable set $P$ of propositional variables,
the \emph{language over $\Sigma$ generated by $P$} is the absolutely free
algebra over $\Sigma$ freely generated by $P$, denoted $\LangAlgA$. 
Its carrier is denoted by $\LangSetA$ and its elements are called \emph{formulas}.
The collection of
all subformulas of a formula
$\FmA$ is denoted by
$\Subf{\FmA}$.
Similarly, the set of all propositional variables
occurring in $\FmA \in \LangSetA$
is denoted by $\Props{\FmA}$.

A
\emph{single-conclusion logic (over $\Sigma$)}
is a  (non-necessarily structural) consequence relation
$\vdash$ on $\LangSetA$
and a \emph{multiple-conclusion logic (over $\Sigma$)}
is a generalized consequence relation 
$\SetSetCR$ on $\LangSetA$
--- see~\cite[Secs. 1.12,~1.16]{humberstone2011} for
precise definitions.
The \emph{single-conclusion companion}
of a given multiple-conclusion logic~$\SetSetCR$
is the single-conclusion logic $\vdash_{\SetSetCR}$ 
such that $\FmSetA \vdash_{\SetSetCR} \FmB$
if, and only if, $\FmSetA \SetSetCR \{\FmB\}$.
We adopt the convention of
omitting curly braces when writing sets of formulas
in statements
involving (generalized) consequence relations.
The complement of a given
$\SetSetCR$,
i.e., 
$\PowerSet\LangSetA \times\PowerSet\LangSetA{\setminus}\SetSetCR{}$,
will be denoted
by~$\NSetSetCR{}$.

Given a signature $\Sigma$, a \emph{three-valued matrix over $\Sigma$} (or \emph{$\Sigma$-matrix})
is a structure $\MatA \SymbDef \langle \AlgA, D \rangle$,
where $\AlgA$ is an algebra over $\Sigma$ with carrier $O_3 \SymbDef \{ \bv,\uv,\tv \}$, and 
$D \subseteq O_3$ (the set of \emph{designated values}). The algebra $\AlgA$ assigns to each
$k$-ary connective in $\Sigma$ a $k$-ary operation $\ConA_\AlgA : O_3^k \to O_3$, called the
\emph{interpretation} of $\ConA$ in $\MatA$.
For each formula $\FmA(p_1,\ldots,p_k)$ on $k$ propositional variables,
we denote by $\FmA_\AlgA$ the $k$-ary derived operation on $\AlgA$ induced by $\FmA$ in the standard way.
Homomorphisms from $\LangAlgA$ to $\AlgA$ are called \emph{$\MatA$-valuations}.
Every $\MatA$ determines multiple-conclusion and single-conclusion 
substitution-invariant consequence relations
in the following way:
\begin{align*}
    \FmSetA \SetSetCR_{\MatA} \FmSetB
    \text{ iff, for no $\MatA$-valuation $v$, }
    v[\FmSetA] \subseteq D
    \text{ and }
    v[\FmSetB] \subseteq O_3{\setminus}D\\
    \FmSetA \SetFmlaCR_{\MatA} \FmB
    \text{ iff, for no $\MatA$-valuation $v$, }
    v[\FmSetA] \subseteq D
    \text{ and }
    v(\FmB) \in O_3{\setminus}D
\end{align*}
We may also consider proper subsets of $\MatA$-valuations, leading to potentially non-structural consequence relations,
as Cooper did for $\CooOL$.
Following~\cite{Cooper1968}, we consider here the particular class that we call \emph{$b\MatA$-valuations},
comprising those valuations that assign to propositional variables either $\bv$ or $\tv$
(
complex formulas may still be assigned the value $\uv$ even under this restriction, depending on the
interpretations of the connectives).
We thus obtain the following logics:
\begin{align*}
    \FmSetA \SetSetCRBiv_{\MatA} \FmSetB
    \text{ iff, for no b$\MatA$-valuation $v$, }
    v[\FmSetA] \subseteq D
    \text{ and }
    v[\FmSetB] \subseteq O_3{\setminus}D\\
    \FmSetA \SetFmlaCRBiv_{\MatA} \FmB
    \text{ iff, for no b$\MatA$-valuation $v$, }
    v[\FmSetA] \subseteq D
    \text{ and }
    v(\FmB) \in O_3{\setminus}D
\end{align*}
We use the above notions to precisely define what
logics we are interested in here.
Let $\Sigma_{\CooOL} \SymbDef \{ \neg, \lor, \land, \to \}$
and $\Sigma$ be a signature such that
$\Sigma \cap \Sigma_{\CooOL} \neq \varnothing$.
We say that a three-valued $\Sigma$-matrix $\MatA$ is an \emph{$\CooOL$-matrix}
when the connectives in $\Sigma \cap \Sigma_{\CooOL}$ are interpreted as
in Figure~\ref{fig:truth-tables2}
and the set of designated values is $\{ \uv,\tv \}$.
If $\Sigma = \Sigma_{\CooOL}$,
$\SetSetCR{}_{\SOL} \SymbDef \SetSetCR{}_{\MatA}$ and $\SetFmlaCR_{\SOL} \,\SymbDef\, \SetFmlaCR_{\MatA}$ are what
we call 
(multiple-conclusion and single-conclusion, respectively) 
$\SOL$;
similarly, 
$\SetSetCRBiv_{\CooOL} \SymbDef \SetSetCRBiv_{\MatA}$ and
$\SetFmlaCRBiv_{\CooOL} \SymbDef \SetFmlaCRBiv_{\MatA}$ go by the name of
(multiple-conclusion and single-conclusion, respectively) 
$\CooOL$.
If $\Sigma \neq \Sigma_{\CooOL}$ instead,
$\SetSetCR{}_{\MatA}$ and $\vdash_{\MatA}$ are called
a (multiple-conclusion and single-conclusion, respectively) \emph{fragment} of $\SOL$
if $\Sigma \subseteq \Sigma_{\CooOL}$
and an \emph{expansion} (of a fragment of) $\SOL$ in
case $\Sigma {\setminus} \Sigma_{\CooOL} \neq \varnothing$.
In case $\Sigma$ is arbitrary but $\MatA$ interprets
its connectives by functions that are term-definable
from the interpretations of $\Sigma_{\CooOL}$ in Figure~\ref{fig:truth-tables2},
we say more generally that
$\SetSetCR{}_{\MatA}$ and $\vdash_{\MatA}$ are \emph{term-definable fragments}
of $\SOL$. The same terminologies for $\CooOL$ are defined analogously.
{We will be particularly interested
in the following term-definable connectives (see Figure~\ref{fig:truth-tables-deriv} for their truth tables):

\begin{equation*}
\begin{aligned}[c]
\Di  x  & : = \nnot x \to x \\
x \imp y & : = \nnot  x \lor  y  \\
x \supset y & : = \Di  x \imp y  \\
\end{aligned}
\qquad \qquad
\begin{aligned}[c]
x \sqcup y & : =  \nnot \Di  \nnot ( x \lor y) \land ((x \to y) \to \Di  y) \\
x \sqcap y & := \nnot (\nnot x \sqcup \nnot y)\\
    x \lordimpneg y & := (x \Rightarrow y) \Rightarrow ((y \Rightarrow x) \Rightarrow x)
\end{aligned}
\end{equation*}}


\begin{figure}[tbh]
    \centering
    \begin{tabular}{@{}c|ccc@{}}
      \toprule
         $\sqcap$ & \uv & \tv & \bv \\
        \midrule
         \uv & \uv & \uv & \bv\\
         \tv & \uv & \tv & \bv \\
         \bv & \bv & \bv & \bv \\
        \bottomrule
    \end{tabular}
    \quad
    \begin{tabular}{@{}c|ccc@{}}
       \toprule
         $\sqcup$ & \uv & \tv & \bv \\
        \midrule
         \uv & \uv & \tv & \uv\\
         \tv & \tv & \tv & \tv \\
         \bv & \uv & \tv & \bv \\
        \bottomrule
    \end{tabular}
    \quad 
    \begin{tabular}{@{}c|c@{}}
        \toprule
         & $\Di$ \\
         \midrule
         \uv & \uv\\
         \tv & \uv\\
         \bv & \bv\\
         \bottomrule
    \end{tabular}
    \quad
    \begin{tabular}{@{}c|ccc@{}}
     \toprule
         $\supset$ & \uv & \tv & \bv \\
        \midrule
         \uv & \uv & \tv & \bv\\
         \tv & \uv & \tv & \bv \\
         \bv & \tv & \tv & \tv \\
        \bottomrule
    \end{tabular}
    \quad
    \begin{tabular}{@{}c|ccc@{}}
         \toprule
          $\imp$ & \uv & \tv & \bv \\
         \midrule
          \uv & \uv & \tv & \bv \\
          \tv & \bv & \tv & \bv\\
          \bv & \tv & \tv & \tv\\
         \bottomrule
     \end{tabular}
     \quad
     \begin{tabular}{@{}c|ccc@{}}
         \toprule
          $\lordimpneg$ & \uv & \tv & \bv \\
         \midrule
          \uv & \uv & \tv & \tv\\
          \tv & \tv & \tv & \tv \\
          \bv & \tv & \tv & \bv \\
         \bottomrule
     \end{tabular}
\caption{Truth tables of term-definable connectives of $\CooOL$/$\SOL$.}
\label{fig:truth-tables-deriv}
\end{figure}
%
%
Now we present definitions related to Hilbert-style calculi.
A \emph{multiple-conclusion Hilbert calculus} over $\Sigma$ is traditionally defined as a collection $\CalcA$
of \emph{multiple-conclusion rules} of the form $\frac{\FmSetA}{\FmSetB}$,
where $\FmSetA,\FmSetB \subseteq \LangSetA$.
Here we also consider rules denoted by $\frac{\FmSetA}{\FmSetB}[\FmSetD]$ for $\FmSetD\subseteq \Props{\FmSetA \cup \FmSetB}$, to be able to axiomatize 
non-substitution-invariant logics, as we shall soon explain. 
We will often identify
$\frac{\FmSetA}{\FmSetB}$
with
$\frac{\FmSetA}{\FmSetB}[\varnothing]$.
Rules with $\FmSetD \neq \varnothing$
are referred to as
\emph{identity-instance rules}.
When writing rules, we usually omit curly braces from the set notation.
A \emph{proof} of $(\FmSetC,\FmSetD)$ in $\CalcA$ is a finite directed rooted tree 
where each node is
labelled either with a set of
formulas or with the discontinuation symbol $\star$, such that (i) the root is labelled with a superset of $\FmSetC$; (ii) every leaf is labelled
either with a set having non-empty intersection with $\FmSetD$ or with $\star$;
(iii) every non-leaf node has children determined by a substitution instance of a rule of inference of $\CalcA$, in the way we now detail.
A rule instance $\RuleA^\sigma$ applies to a node $\NodeA$
when
the antecedent of $\RuleA^\sigma$ is contained in the label of $\NodeA$;
the application results in $\NodeA$ having
exactly one child node for
each formula $\FmB$ in the succedent of $\RuleA^\sigma$,
which is, in turn, labelled with the same formulas as those of (the label of) $\NodeA$
plus $\FmB$. In case $\RuleA^\sigma$ has an empty
succedent, then $\NodeA$ has a single child node
labelled with $\Star$.
If $\RuleA$ has the form $\frac{\FmSetA}{\FmSetB}$, then any substitution
$\sigma$ can be applied to instantiate it; if it is of the form $\frac{\FmSetA}{\FmSetB}[\FmSetD]$,
however, 
only substitutions $\sigma$ with
$\sigma(p) = p$ for each $p \in \FmSetD$
may be applied.
When we display proof trees, it is common to write as labels only the
formulas introduced by the rule application, instead of the whole accumulated set of formulas.
Examples of multiple-conclusion proofs 
may be found in Figure~\ref{fig:sol-derivations}.

We write $\FmSetC \SetSetCR_{\CalcA} \FmSetD$ whenever there is a proof of $(\FmSetC,\FmSetD)$
in $\CalcA$, and write $\SetFmlaCR_{\CalcA}$ for the single-conclusion companion of $\SetSetCR_{\CalcA}$.
These relations are respectively multiple-conclusion and single-conclusion consequence relations
that are substitution-invariant when no identity-instance rule is present in $\CalcA$.
Single-conclusion Hilbert calculi are just the traditional Hilbert calculi, which
can be seen as multiple-conclusion calculi in which only rules of the form
$\frac{\FmSetA}{\FmB}$ and $\frac{\FmSetA}{\FmB}[\FmSetD]$ are allowed,
where $\FmSetA \subseteq \LangSetA$ and $\FmB \in \LangSetA$.
Derivations then can be seen as linear trees, usually displayed simply as sequences of formulas.
We say that a Hilbert calculus $\CalcA$  (multiple- or single-conclusion) \emph{axiomatizes} $\SetSetCR{}$
if $\SetSetCR{} = \SetSetCR_{\CalcA}$. It axiomatizes $\SetFmlaCR$ in case
$\SetFmlaCR \,=\, \SetFmlaCR_{\CalcA}$.

\section{Hilbert-style axiomatizations for $\SOL$ and $\CooOL$}
\label{sec:axiomatization}
We present now, in a modular way,
multiple- and single-conclusion
Hilbert-style axiomatizations
for $\SOL$ and  $\CooOL$
and  for fragments/expansions
in which $\neg$ is present.

\subsection{Multiple-conclusion}

Every logic determined by a finite 
matrix is finitely axiomatized by a multiple-conclusion calculus~\cite{ss1978}.
Moreover, if the matrix is \emph{monadic},
the calculus satisfies a generalized form of
analyticity and can be effectively generated from the matrix description~\cite{marcelino19woll}.
We now describe in more detail these notions.

A matrix $\MatA \SymbDef \langle \AlgA, D \rangle$ is \emph{monadic} if there is a unary formula
$\Sep(\PropA) \in \LangSetAp$,
sometimes called a \emph{separator}, for each pair of truth values $x, y \in A$, such that
$\Sep_{\AlgA}(x) \in D$ and $\Sep_{\AlgA}(y) \in \ValuesSetComp{A}{D}$
or vice-versa.
A multiple-conclusion Hilbert calculus $\CalcA$ is \emph{$\Theta$-analytic}, for
$\Theta(p)$ a set of unary formulas, whenever $\FmSetA \SetSetCR_\CalcA \FmSetB$
is witnessed by a derivation using only subformulas of
$\FmSetA \cup \FmSetB$ or formulas in $\bigcup_{\FmC \in \Subf{\FmSetA\cup\FmSetB}} \Theta(\FmC)$.
{For example,
if $\Theta(p) = \{ p, \neg p \}$
and $\CalcA$ is $\Theta$-analytic,
checking whether
$r, q \land r \,\SetSetCR_{\CalcA}\, p \to r, \neg q$ amounts to looking only for
derivations in which formulas in
$\{ r, q, p, q \land r, p \to r, \neg q, \neg r, \neg p, \neg(q \land r), \neg(p \to r), \neg\neg q \}$ occur.}
{Because in an OL-matrix
    $p$ separates $(\tv,\bv)$ and $(\uv,\bv)$,
    while $\neg p$ separates $(\tv,\uv)$,
    the following holds.}


\begin{lemma}
    \label{lem:monadicity-of-matrices}
    Any $\CooOL$-matrix over 
    $\Sigma \supseteq \{ \neg \}$
    is monadic, with set of separators $\{ p, \neg p \}$.
\end{lemma}

Lemma~\ref{lem:monadicity-of-matrices} implies 
that all multiple-conclusion fragments/expansions of 
$\SOL$ containing negation
admit $\{ p, \neg p \}$-analytic multiple-conclusion
Hilbert-style axiomatizations generated from the matrix description
in a modular way, as the next theorem states.
For a detailed presentation of the
calculi generation for
three-valued logics {and for the associated adequacy proofs}, see~\cite[Sec.~5]{threevalchapter}.

\begin{theorem}\label{the:mc-axiomat}
    Let $\Sigma$ be a signature such that $\neg \in \Sigma$
    and
    $\MatA$ be an $\CooOL$-matrix over $\Sigma$.
    Then $\SetSetCR_{\MatA}$
    is axiomatized
    by the $\{ p, \neg p\}$-analytic calculi
    \[
    \CalcA_\Sigma \SymbDef \bigcup_{\ConA \,\in \Sigma} \CalcA_\ConA,
    \]
    where each set $\CalcA_\ConA$, $\ConA \in \Sigma$,
    is generated by the procedure described in~\cite{threevalchapter}.
    For the connectives of $\CooOL$,
    these sets 
    are displayed in Figure~\ref{fig:mc-rules-connectives}.
\end{theorem}

\begin{figure}[tbh!]
 $\CalcA_\neg$
 $$\frac{p}{\neg \neg p}{\mathsf{r}^\neg_1} \qquad \frac{\neg \neg p}{p}{\mathsf{r}^\neg_2}\qquad
 \frac{}{p, \neg p}{\mathsf{r}^\neg_3}$$
 
\rule{\textwidth}{0.2pt}

$\CalcA_\to$
$$\frac{p,p\to q}{q}{\mathsf{r}^\to_1}
\qquad 
\frac{q}{p\to q}{\mathsf{r}^\to_2}
\qquad
\frac{}{p,p\to q}{\mathsf{r}^\to_3}$$

$$\frac{p,\neg(p\to q)}{\neg q}{\mathsf{r}^\to_4}
\qquad 
\frac{\neg q}{\neg(p\to q)}{\mathsf{r}^\to_5}
\qquad
\frac{}{p,\neg(p\to q)}{\mathsf{r}^\to_6}$$

\rule{\textwidth}{0.2pt}

$\CalcA_\lor$
    $$\frac{}{\neg p, p \lor q} {\mathsf{r}^\lor_1}
    \qquad 
    \frac{}{\neg q, p \lor q}{\mathsf{r}^\lor_2} \qquad 
    \frac{\neg(p \lor q)}{\neg p}{\mathsf{r}^\lor_3} \qquad 
    \frac{\neg(p \lor q)}{\neg q}{\mathsf{r}^\lor_4}
    \qquad 
    \frac{\neg(p \lor q), p \lor q}{p}{\mathsf{r}^\lor_5}$$
    $$\frac{\neg(p \lor q), p \lor q}{q}{\mathsf{r}^\lor_6}
    \qquad
    \frac{\neg p, \neg q}{\neg(p \lor q)}{\mathsf{r}^\lor_7}
    \qquad
    \frac{\neg p, p \lor q}{q}{\mathsf{r}^\lor_8} \qquad 
    \frac{\neg q, p \lor q}{p}{\mathsf{r}^\lor_9} \qquad 
    \frac{p \lor q}{p,q}{\mathsf{r}^\lor_{10}} \qquad 
    \frac{p, q}{p \lor q}{\mathsf{r}^\lor_{11}}$$

\rule{\textwidth}{0.2pt}

$\CalcA_\land$
$$\frac{p\land q}{p}{\mathsf{r}^\land_1}
\qquad 
\frac{p\land q}{q}{\mathsf{r}^\land_2}
\qquad
\frac{p,q}{p\land q}{\mathsf{r}^\land_3}$$
$$\frac{p\land q,\neg(p\land q)}{\neg p}{\mathsf{r}^\land_4}\qquad 
\frac{p\land q,\neg(p\land q)}{\neg q}{\mathsf{r}^\land_5}\qquad
\frac{\neg p,\neg q}{\neg(p\land q)}{\mathsf{r}^\land_6}
\qquad
\frac{p}{p\land q,\neg q}{\mathsf{r}^\land_7}
\qquad
\frac{q}{p\land q,\neg p}{\mathsf{r}^\land_8}
$$
%
\caption{Multiple-conclusion rules for the connectives of $\SOL$.}
\label{fig:mc-rules-connectives}
\end{figure}

 For examples
of derivations in the multiple-conclusion calculus
for $\SOL$ obtained from the above theorem, see Figure~\ref{fig:sol-derivations}.
%
The next theorem 
presents axiomatizations (recall, not substitution-invariant) for the corresponding
multiple-conclusion
fragments and expansions of $\CooOL$.

%

\begin{theorem}
    Let $\Sigma$ be a signature such that $\neg \in \Sigma$
    and $\MatA$ be an $\CooOL$-matrix over $\Sigma$.
    Then $\SetSetCRBiv_{\MatA}$ is axiomatized by
    \[\CalcA_\Sigma^{\CooOL} \SymbDef \CalcA_\Sigma \cup \left\{ \frac{p, \neg p}{\varnothing}[p] : p \in P \right\}.\]
\end{theorem}
\begin{proof}
    The nontrivial inclusion
    is
    $\SetSetCRBiv_{\MatA} \subseteq \SetSetCR{}_{\CalcA_\Sigma^{\CooOL}}$. 
    Following the standard completeness proofs exemplified
    in~\cite[Sec. 5.2]{threevalchapter},
    from 
    $\FmSetA \NSetSetCR{\CalcA_{\Sigma}^\CooOL} \FmSetB$ 
    we
    construct a valuation $v$
    witnessing $\FmSetA\NSetSetCRBiv{\MatA}\FmSetB$.
    In that construction,
    the rules
    $\frac{p,\neg p}{\varnothing}[p]$
    force that the valuation only 
    assigns values in $\{ \bv, \tv \}$
    to propositional variables
    and the fact that $ \CalcA_{\Sigma}
    \subseteq
    \CalcA_{\Sigma}^\CooOL$ and $\CalcA_\Sigma$ axiomatizes
    $\SetSetCR_{\MatA}$ guarantees the homomorphism 
    requirement on $v$.
\end{proof}

The reader may appreciate the importance of the extra rules in the above theorem by proving
$p \lor (q \to r) \vdash_{\CooOL} p \lor r$,
even though
$p \lor (q \to r) \not\vdash_{\SOL} p \lor r$ (easy to see semantically). 
More examples like this may be found in~\cite{Cooper1968}.


\subsection{Single-conclusion}
\label{ss:single}

When a suitable connective is available, 
a multiple-conclusion calculus 
may be effectively translated into a
single-conclusion calculus for its single-conclusion companion~\cite{ss1978}.
We now define what is a suitable connective
and the
calculi translations.

\begin{definition}\label{def:disj-imp}
    Let $\vdash$ be a single-conclusion logic over $\LangSetA$ and $\ConA$ be a derived connective in $\LangSetA$.
Then $\ConA$ is
    \begin{enumerate}
        \item a \emph{disjunction} in $\vdash$ whenever $\Gamma, \FmA \,\ConA\, \FmB \vdash \FmC$ iff $\Gamma,\FmA \vdash \FmC$ and $\Gamma,\FmB \vdash \FmC$,
    for all $\Gamma \cup \{ \FmA,\FmB,\FmC \} \subseteq \LangSetA$.
        \item an \emph{implication} in $\vdash$ whenever 
        {$\Gamma \vdash \FmA \,\ConA\, \FmB$
        iff
        $\Gamma,\FmA \vdash \FmB$},
    for all $\Gamma \cup \{ \FmA,\FmB \} \subseteq \LangSetA$.
    \end{enumerate}
\end{definition}
\noindent 
{Note that, in standard terminology,
an implication in $\vdash$ is a binary connective
satisfying the Deduction-Detachment Theorem (DDT).}
In what follows, given a set of formulas $\FmSetA$
{and a binary connective $\ArbDisj$}, let $\FmSetA \ArbDisj \FmB \SymbDef \{ \FmA \ArbDisj \FmB \mid \FmA \in \FmSetA \}$
and $\bigoplus \{ \FmA_1,\ldots,\FmA_m \} \SymbDef \FmA_1 \ArbDisj (\FmA_2 \ArbDisj \ldots (\ldots \ArbDisj \FmA_n)\ldots)$.

\begin{definition}\label{def:trans-disj-mc-to-sc}
Let $\mathsf{R}$ be
a multiple-conclusion calculus
{and $\ArbDisj$ be a binary connective}.
We define $\mathsf{R}^{\ArbDisj}$ as the
    single-conclusion calculus
\[
\left\{\frac{p \ArbDisj p}{p},\frac{p}{p \ArbDisj q}, \frac{p \ArbDisj q}{q \ArbDisj p}, \frac{p \ArbDisj (q \ArbDisj r)}{(p \ArbDisj q) \ArbDisj r}\right\} \cup
\left\{\mathsf{r}^\ArbDisj \mid \mathsf{r} \in \mathsf{R}\right\}\]
where $\mathsf{r}^\ArbDisj$
is $\frac{\varnothing}{\bigoplus\FmSetB}[\FmSetD]$
    if $\mathsf{r} = \frac{\varnothing}{\FmSetB}[\FmSetD]$, $\frac{\FmSetA \ArbDisj p_0}{(\bigoplus \FmSetB) \ArbDisj p_0}[\FmSetD]$
    if $\mathsf{r} = \frac{\FmSetA}{\FmSetB}[\FmSetD]$,
    and $\frac{\FmSetA \ArbDisj p_0}{p_0}[\FmSetD]$
    if $\mathsf{r} = \frac{\FmSetA}{\varnothing}[\FmSetD]$,
    for $p_0$ a
    propositional variable
    not occurring in $\RuleA$.
\end{definition}

Now we move to translations when an implication is present.
Let $p_0$ be a propositional variable not occurring in 
$\FmA_1,\ldots,\FmA_m,\FmB_1,\ldots,\FmB_n$
and {$\ArbImp$ be a binary connective}.
Define
$\{ \FmA_1,\ldots,\FmA_m \} \ArbImp \{ \FmB_1,\ldots,\FmB_n \} \SymbDef \FmA_1 \ArbImp (\{ \FmA_2,\ldots,\FmA_m \} \ArbImp \{ \FmB_1,\ldots,\FmB_n \})$,
$\varnothing \ArbImp \{ \FmB_1,\ldots,\FmB_n \} \SymbDef (\FmB_1 \ArbImp p_0) \ArbImp (\varnothing \ArbImp \{ \FmB_2,\ldots,\FmB_{n} \})$
and $\varnothing \ArbImp \varnothing \SymbDef p_0$.
For example, $\{ p, q \} \ArbImp \{ p \land q \} = p \ArbImp (q \ArbImp (((p \land q) \ArbImp p_0 ) \ArbImp p_0))$.

\begin{definition}
Let $\mathsf{R}$ be
a multiple-conclusion calculus
{and $\ArbImp$ be a binary connective}.
We define $\mathsf{R}^{\ArbImp}$ as the
    single-conclusion calculus containing
    all rules and axioms of intuitionistic implication (where $\ArbImp$ is taken as this implication)
    and axioms of the form
    \[
        \frac{}{
            \{ \FmA_1,\ldots,\FmA_m \} \ArbImp \{ \FmB_1,\ldots,\FmB_n \}
        }[\FmSetD] 
    \]
    for each rule $\frac{\FmA_1,\ldots,\FmA_m}{\FmB_1,\ldots,\FmB_n}[\FmSetD]$ belonging to $\CalcA$.
\end{definition}

The following theorem establishes that, when $\ArbDisj$ and $\ArbImp$ are
respectively a disjunction and an implication as previously defined,
the above translations produce (finite) single-conclusion axiomatizations 
from (finite) multiple-conclusion axiomatizations.
{In its original formulation,
rules of the general form $\frac{\FmSetA}{\FmSetB}[\FmSetD]$
were not considered, but their addition does not
invalidate the result.}

\begin{theorem}[{\cite[Thm. 5.37, Lem. 19.20]{ss1978}}]
    \label{the:translate-mc-sc}
    Let $\CalcA$
    be a multiple-conclusion calculus.
        (1) If $\ArbDisj$ is a disjunction in $\vdash_{\CalcA}$, then
            $\vdash_{{\CalcA}} \,=\, \vdash_{\CalcA^{\ArbDisj}}$.
        (2) If $\ArbImp$ is an implication in $\vdash_{\CalcA}$, then
            $\vdash_{{\CalcA}} \,=\, \vdash_{\CalcA^{\ArbImp}}$.
\end{theorem}

It is not hard to check that $\to$ is an implication  in $\SOL$
(see also Section~\ref{sec:alg}); this gives us  the following single-conclusion axiomatizations.

\begin{theorem}
    \label{the:imp-frags-set-fmla}
    Let $\Sigma$ be a signature such that $\{ \neg,\to \} \subseteq \Sigma$,
    and let $\MatA$ be an $\CooOL$-matrix over $\Sigma$. 
    Then $\vdash_\MatA$ is axiomatized by
    $\CalcA_\Sigma^\to$,
    where $\CalcA_\Sigma$ is given as in Theorem~\ref{the:mc-axiomat}.
\end{theorem}

More interestingly,  we may replace $\to$ by either $\land$ or $\lor$,
providing axiomatizations for more fragments of $\SOL$.
Neither $\land$ nor $\lor$ is a disjunction or
an implication in the above sense, but each of them allows us
to define a connective that is a disjunction and thus suitable
for the multiple-conclusion to single-conclusion translation.

\begin{theorem}
    \label{the:no-imp-frags-set-fmla}
    Let $\Sigma$ be a signature such that either $\{ \neg,\land \} \subseteq \Sigma$
    or $\{\neg,\lor \} \subseteq \Sigma$
    and let $\MatA$ be an $\CooOL$-matrix over $\Sigma$. 
    Then $\vdash_{\MatA}$ is axiomatized
    by
    $\CalcA_\Sigma^{\lordimpneg}$,
    where $\CalcA_\Sigma$ is given as in Theorem~\ref{the:mc-axiomat},
    and $p \lordimpneg q := (p \Rightarrow q) \Rightarrow ((q \Rightarrow p) \Rightarrow p)$.
\end{theorem}
\begin{proof}
    Let 
    $p \lordimpneg q := (p \Rightarrow q) \Rightarrow ((q \Rightarrow p) \Rightarrow p)$
    and recall that $\Rightarrow$ was defined as $\neg p \lor q$ (thus using only $\neg$ and $\lor$).
    Note that it could also have been defined using $\neg$ and $\land$, since $\lor$ is definable
    from $\neg$ and $\land$. 
    From the truth table of $\lordimpneg$ displayed in Figure~\ref{fig:truth-tables2},
    it is easy to see that $\lordimpneg$ is a disjunction in the sense of Definition~\ref{def:disj-imp}.
\end{proof}

We display in Figure~\ref{fig:sc-rules-connectives-disj} 
the axiomatization of $\SOL$ following the previous theorem.

\begin{remark}
    We could have used 
    the connective $\sqcup$ defined before
    instead of $\lordimpneg$.
\end{remark}

\begin{remark}
    Since $\lordimpneg$ was defined using only
     $\Rightarrow$,
    this result (and the next) also
    applies to the  term-definable single-conclusion fragments $\{ \neg,\Rightarrow \}$
    of $\CooOL$ and $\SOL$.
\end{remark}

\begin{figure}[tbh!]
    \centering
    \scalebox{.9}{
    \begin{tikzpicture}[baseline]
        \node (0) {$\varnothing$};
        \node (1) [below left=.5cm and .5cm of 0] {$\FmA\to\FmB$};
        \node (2) [below right=.5cm and .5cm of 0] {$(\FmA\to\FmB) \to \neg(\FmA \to \neg\FmB)$};
        \node (3) [below left=.5cm and .5cm of 1] {$\FmA$};
        \node (4) [below right=.5cm and .5cm of 1] {$\neg(\FmA \to \neg\FmB)$};
        \node (5) [below of=4] {$(\FmA\to\FmB) \to \neg(\FmA \to \neg\FmB)$};
        \node (6) [below of=3] {$\FmB$};
        \node (7) [below of=6] {$\neg\neg\FmB$};
        \node (8) [below of=7] {$\neg(\FmA \to \neg\FmB)$};
        \node (9) [below of=8] {$(\FmA\to\FmB) \to \neg(\FmA \to \neg\FmB)$};
        \draw (0) -- (1) node[midway,above left] {$\mathsf{r}_3^\to$};
        \draw (0) -- (2) node[midway,above left] {};
        \draw (1) -- (3) node[midway,above left] {$\mathsf{r}_6^\to$};
        \draw (1) -- (4) node[midway,above left] {};
        \draw (4) -- (5) node[midway, left] {$\mathsf{r}_2^\to$};
        \draw (3) -- (6) node[midway,left] {$\mathsf{r}_1^\to$};
        \draw (6) -- (7) node[midway, left] {$\mathsf{r}_1^\neg$};
        \draw (7) -- (8) node[midway, left] {$\mathsf{r}_5^\to$};
        \draw (8) -- (9) node[midway, left] {$\mathsf{r}_2^\to$};
    \end{tikzpicture}
    \begin{tikzpicture}[baseline]
        \node (0) {$\varnothing$};
        \node (11) [below left=.5cm and .2cm of 0]  {$\neg\FmA$};
    \node (12) [below right=.5cm and .2cm of 0] {$\neg(\neg\FmA \to \FmA)$};
        \node (22) [below of=11] {$\neg(\neg \FmA \to \FmA)$};
        \draw (0)  -- (11) node[midway,above left] {$\mathsf{r}^\to_6$};
        \draw (0)  -- (12) node[midway,above left] {};
        \draw (11) -- (22) node[midway, left] {$\mathsf{r}^\to_5$};
    \end{tikzpicture}
    }
    \caption{Derivations of Boethius' thesis and Aristotle's thesis in $\SOL$ (see Figure~\ref{fig:mc-rules-connectives}).}
    \label{fig:sol-derivations}
\end{figure}
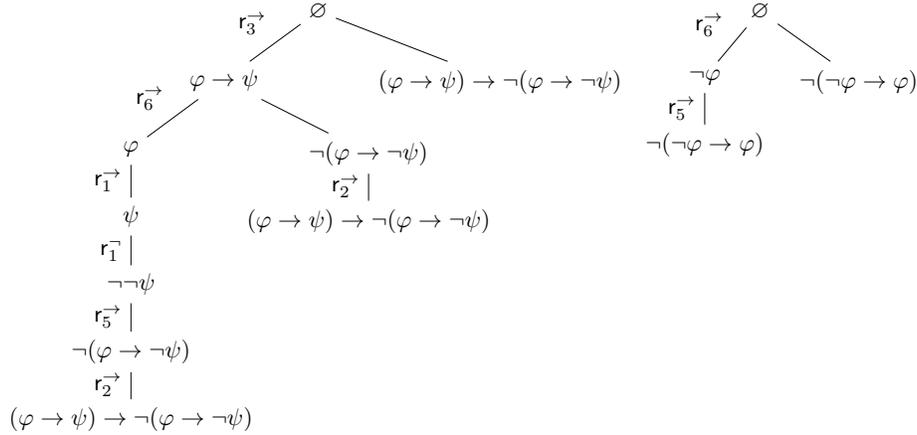

\begin{figure}[tbh!]
 $\CalcA_{\lordimpneg}$
 $$
    \frac{p \lordimpneg p}{p}
    \qquad
    \frac{p}{p \lordimpneg q}
    \qquad
    \frac{p \lordimpneg q}{q \lordimpneg p}
    \qquad
    \frac{p \lordimpneg (q \lordimpneg r)}{(p \lordimpneg q) \lordimpneg r}
 $$
\rule{\textwidth}{0.2pt}

 $\CalcA_\neg$
 $$\frac{p \lordimpneg r}{\neg \neg p \lordimpneg r}\qquad \frac{\neg \neg p \lordimpneg r}{p \lordimpneg r}\qquad
    \frac{}{p \lordimpneg \neg p }$$
    
\rule{\textwidth}{0.2pt}

$\CalcA_\to$
    $$\frac{p \lordimpneg r, (p\to q) \lordimpneg r}{q \lordimpneg r}\qquad 
    \frac{q \lordimpneg r}{(p\to q)\lordimpneg r}\qquad
    \frac{}{p \lordimpneg (p\to q)}$$

    $$\frac{p \lordimpneg r,\neg(p\to q) \lordimpneg r}{\neg q \lordimpneg r}\qquad 
    \frac{\neg q \lordimpneg r}{\neg(p\to q) \lordimpneg r}
    \qquad
    \frac{}{p \lordimpneg \neg(p\to q)}$$

\rule{\textwidth}{0.2pt}

$\CalcA_\lor$
    $$\frac{}{\neg p \lordimpneg (p \lor q)} 
    \qquad 
    \frac{}{\neg q \lordimpneg  (p \lor q)}
    \qquad
    \frac{\neg(p \lor q) \lordimpneg r}{\neg p \lordimpneg r} 
    \qquad 
    \frac{\neg(p \lor q) \lordimpneg r}{\neg q \lordimpneg r}$$

    $$\frac{\neg(p \lor q) \lordimpneg r, (p \lor q) \lordimpneg r}{p \lordimpneg r} 
    \qquad 
    \frac{\neg(p \lor q) \lordimpneg r, (p \lor q) \lordimpneg r}{q \lordimpneg r} 
    \qquad 
    \frac{\neg p \lordimpneg r, \neg q \lordimpneg r}{\neg(p \lor q) \lordimpneg r}$$

    $$\frac{\neg p \lordimpneg r, (p \lor q) \lordimpneg r}{q \lordimpneg r} 
    \qquad 
    \frac{\neg q \lordimpneg r, (p \lor q) \lordimpneg r}{p \lordimpneg r} 
    \qquad 
    \frac{(p \lor q) \lordimpneg r}{(p \lordimpneg q) \lordimpneg r} 
    \qquad 
    \frac{p \lordimpneg r, q \lordimpneg r}{(p \lor q) \lordimpneg r}$$
    
\rule{\textwidth}{0.2pt}

$\CalcA_\land$
    $$\frac{(p\land q) \lordimpneg r}{p \lordimpneg r}\qquad \frac{(p\land q) \lordimpneg r}{q \lordimpneg r}
    \qquad
    \frac{p \lordimpneg r, q \lordimpneg r}{(p\land q) \lordimpneg r}$$

    $$\frac{(p\land q) \lordimpneg r, \neg(p\land q) \lordimpneg r}{\neg p \lordimpneg r}
    \qquad 
    \frac{(p\land q) \lordimpneg r,\neg(p\land q) \lordimpneg r}{\neg q \lordimpneg r}
    $$

    $$
    \frac{\neg p \lordimpneg r,\neg q \lordimpneg r}{\neg(p\land q) \lordimpneg r}
\qquad
    \frac{p \lordimpneg r}{((p\land q) \lordimpneg \neg q) \lordimpneg r}
\qquad
    \frac{q \lordimpneg r}{((p\land q) \lordimpneg \neg p) \lordimpneg r}
$$
%
\caption{Single-conclusion calculus for single-conclusion 
$\SOL$ produced via Definition~\ref{def:trans-disj-mc-to-sc}.
By Theorem~\ref{the:no-imp-frags-set-fmla}, one can modularly add suitable rules
to $\CalcA_{\lordimpneg} \cup \CalcA_\neg \cup \CalcA_\lor$ to axiomatize fragments/expansions of $\SOL$ 
over signatures $\Sigma \supseteq \{ \neg,\lor \}$ ($\lor$ may be replaced by $\land$).}
\label{fig:sc-rules-connectives-disj}
\end{figure}

With the above results, we axiomatized via single-conclusion calculi all single-conclusion fragments/expansions of $\SOL$ containing $\neg$.
For some of them, and for $\SOL$ itself, two calculi were presented. 
The ones obtained from Theorem~\ref{the:imp-frags-set-fmla} have one rule and many axioms,
while the ones from Theorem~\ref{the:no-imp-frags-set-fmla} tend to be rich in rules and have only a few axioms.
We now proceed to extend these calculi to axiomatize the corresponding fragments and expansions of $\CooOL$.

\begin{theorem}
    \label{the:cooper-set-fmla-axiomat}
    Let $\Sigma$ be a signature either with $\{ \neg,\to \} \subseteq \Sigma$, $\{ \neg,\lor \} \subseteq \Sigma$
    or $\{ \neg,\land \} \subseteq \Sigma$, and let $\MatA$ be an $\CooOL$-matrix over $\Sigma$.
    Then $\SetFmlaCRBiv_{\MatA} \,=\, \vdash_{(\CalcA_\Sigma^{\CooOL})^\ConA}$,
    where $\ConA =\, \to$ if $\to \,\in \Sigma$ and $\ConA = \lordimpneg$ otherwise.
\end{theorem}

\section{Algebraic semantics}
\label{sec:alg}
We denote by
$\OL_3$
the three-element algebra
whose operation tables are given in Figure~\ref{fig:truth-tables2}, 
viewed as an algebra in the language $\{ \land, \lor, \to, \nnot \}$.
Denote by $\mathbb{OL}$ the quasi-variety generated by $\OL_3$ (as we shall prove,
 $\mathbb{OL}$ is in fact a variety), and let
$\vDash_{\mathbb{OL}}$ denote the corresponding relative equational consequence relation.


\paragraph{Algebraizability of $\SOL$.}
%
{The matrix semantics
of $\SOL$
makes
it easy to 
check that $\SOL$ is algebraizable in the sense of Blok and Pigozzi~\cite{BP89}.}  
In what follows, 
abbreviate 
$\abs{\alpha} := \alpha \imp \alpha$.

\begin{theorem}
\label{thm:olisalg}
    {
    $\SOL$}
    is
    algebraizable  with translations
$\tau \colon x \mapsto x \approx \abs{x} $
and
$\rho \colon \phi \approx \psi \mapsto 
\{ \phi \imp \psi, \psi \imp \phi 
\}$. (Alternatively, one may take
$\tau \colon x \mapsto x \approx x \to x$
and
$\rho \colon \phi \approx \psi \mapsto 
\{ \phi \to \psi, \psi \to \phi, \nnot \phi \to \nnot \psi, \nnot \psi \to \nnot \phi 
\}$.)
\end{theorem}
\begin{proof}
{Observe (using Figure~\ref{fig:truth-tables-deriv}) that a valuation $v$ over the matrix of $\SOL$ satisfies
an equation $\FmA \approx \FmA \imp \FmA$
iff $v(\FmA) \in \{ \uv,\tv\}$.
Thus,
for all formulas $\Gamma, \phi$,
we have 
$\Gamma \vdash_{\SOL} \phi$ 
iff $\tau(\Gamma) \vDash_{\mathbb{OL}} \tau(\phi)$.}
This is condition (ALG1) of 
algebraizability~\cite[Def.~3.11]{F16}.
To establish algebraizability, it remains to prove (ALG4), i.e.,~that every equation $\phi \approx \psi$ 
is inter-derivable in  $\vDash_{\mathbb{OL}}$ with  $\tau (\rho (\phi \approx \psi))$, which is the set
$\{ \phi \imp \psi \approx \abs{\phi \imp \psi},  \psi \imp \phi \approx \abs{\psi \imp \phi} \}$. 
This is  easily verified in $\OL_3$.
\end{proof}




\paragraph{Algebraic counterpart of  $\SOL$.}

By inspection, $\OL_3$ (Figure~\ref{fig:truth-tables2}) suggests that:
\begin{enumerate}
\item The tables of $\sqcap, \sqcup$ are precisely those of the conjunction and disjunction of 
strong Kleene logic and of G.~Priest's Logic of Paradox,
both defined over the language $\{ \sqcap, \sqcup, \nnot \}$ (see~\cite{threevalchapter} for 
further background on these logics). 

\item Thus, the Logic of Paradox (which also  has $\{\uv, \tv\}$ as designated set, whereas
strong Kleene has $\{ \tv\}$ alone) may be viewed as a definable subsystem
of $\SOL$. 

\item If we consider the language $\{ \sqcap, \sqcup, \supset, \nnot  \}$ then we have the connectives
of Da Costa and D'Ottaviano's three-valued logic $\mathcal{J}3$ (whose designated
set is also $\{\uv, \tv\}$) minus the truth constants. 
These, which are otherwise not definable in $\SOL$, 
need to be 
explicitly included
 in the language; 
we may then define the modal operator of $\Di$ of $\mathcal{J}3$ by $\Di x :  = x \land \tv$.

\item We note that, if we add either $\tv$ or $\bv$ to the language of $\SOL$, then every possible three-valued connective
becomes definable. This is a consequence of the observation that 
$\OL_3$ then becomes a  primal algebra
(Theorem~\ref{the:quasi-var-gen}).
\end{enumerate}

The following proposition is also a matter of straightforward computations.

\begin{proposition}
\label{prop:ol3asn4}
Let $\OL_3 = \la O_3; \land, \lor, \to, \nnot \ra $
be endowed with the above-defined operations (Figure~\ref{fig:truth-tables-deriv}). 
\begin{enumerate}
\item $ \la O_3; \land, 
 \uv, \bv \ra $
is a meet semilattice with 
$\uv$ as maximum and $\bv$ as minimum. 

\item $ \la O_3; \lor, \tv, \uv \ra $
is a join semilattice with 
$\tv$ as maximum and $\uv$ as minimum.

\item $ \la O_3; \sqcap, \sqcup, \tv, \bv 
\ra $ is a 
lattice
with $\tv$ 
 as maximum and $\bv$ as minimum. 



\end{enumerate}
\end{proposition}

For all unexplained universal algebraic terminology, we refer the reader to~\cite{BuSa00}.

\begin{theorem}
    \label{the:quasi-var-gen}
    Denote by $V(\OL_3)$ 
    the variety 
    generated by $\OL_3$. 
\begin{enumerate}
\item $V(\OL_3)$ is  both congruence-distributive and congruence-permutable (i.e., it is an arithmetical variety).
\item $V(\OL_3) = \mathbb{OL} 
$.
\item  $\OL_3$ is quasi-primal,  hence 
$\mathbb{OL}$
is a discriminator variety.
\item If we add either $\bv$ or $\tv$ as a constant to $\OL_3$, then the latter becomes a primal algebra
(where every $n$-ary function for $n\geq 1$ is representable by a term). 
\end{enumerate}
\end{theorem}
\begin{proof}
1. Congruence-distributivity follows from the observation that $\OL_3$ has a term-definable
lattice structure (item (iii) of Proposition~\ref{prop:ol3asn4}). Congruence-permutability is witnessed by the Maltsev term $p(x,y,z)$ defined as follows (cf.~\cite[Thm.~4.10]{QN4}):
$
p(x,y,z)  := (((x \imp y ) \sqcap (z \imp z) ) \imp z) \sqcap (((z \imp y ) \sqcap (x \imp x) ) \imp x)
$.

2. Since $V(\OL_3)$ is congruence-distributive, by item (iii) of~\cite[Thm.~3.6]{ClDa98}) it suffices to verify
that $HS(\OL_3) \subseteq IS(\OL_3)$, which is very easy  ($\OL_3$ has only one proper subalgebra with $\{\uv\}$ as universe, and no non-trivial homomorphic images).

3. Taking into account  item (i) and the fact that  $\OL_3$ is hereditarily simple, apply Pixley's characterization~\cite[Thm.~IV.10.7]{BuSa00} to conclude that $\OL_3$ is quasiprimal. 

4. If we further add one of the non-definable constants ($\bv$ or $\tv$) to $\OL_3$, then by~\cite[Cor.~10.8]{BuSa00}
we obtain a primal algebra.
\end{proof}

\paragraph{Axiomatizing $\mathbb{OL}$.}
{The algebraizability result (Theorem~\ref{thm:olisalg}) can be used to obtain a presentation of the quasi-variety
$\mathbb{OL}$ in the standard way
(see~\cite[Prop.~3.44]{F16}), as well as of the subreducts of $\mathbb{OL}$
corresponding to algebraizable fragments of $\SOL$  (e.g.~those capable of expressing either $\imp$ or $\to$ and $\neg$,
relying on the axiomatizations
obtained in Theorems \ref{the:imp-frags-set-fmla}--\ref{the:cooper-set-fmla-axiomat}).
Moreover, one may obtain a 
more standard Hilbert presentation for
$\SOL$
by 
directly proving that the logic determined by the following calculus is
algebraizable (with the same translations $\tau, \rho$ of Theorem~\ref{thm:olisalg}), and that its 
equivalent semantics is
$\mathbb{OL}$. } This is straightforward, but requires a number of derivations that we
omit due to space limitations. 
The calculus is given by the
following axiom schemata, with \emph{modus ponens} 
(from $\phi$ and $\phi \to \psi$ infer $\psi$)
as the only inference rule
(we abbreviate $\alpha \leftrightarrow \beta := (\alpha \to \beta) \land (\beta \to \alpha)$):
\begin{enumerate}[(HOL1)]
\item \qquad \label{Itm:HOL1} $\phi \to (\psi \to \phi)$
\item \qquad \label{Itm:HOL2} $(\phi \to (\psi \to \gamma)) \to ((\phi \to \psi) \to (\phi \to \gamma))$
\item \qquad \label{Itm:HOL3} $((\phi \to \psi) \to  \phi) \to \phi $
\item \qquad \label{Itm:HOL4} $(\phi \land \psi) \to \phi$
\item \qquad \label{Itm:HOL5} $(\phi \land \psi) \to \psi$
\item \qquad  \label{Itm:HOL6} $(\phi \to \psi) \to ((\phi \to \gamma) \to (\phi \to (\psi \land\gamma)))$
\item \qquad \label{Itm:HOL7} $\nnot \nnot \phi \leftrightarrow \phi$
\item \qquad \label{Itm:HOL7b} $ (\phi \to \nnot \phi)  \to \nnot \phi$
\item \qquad  \label{Itm:HOL9} $ ((\phi \to \nnot \psi) \land (\psi \to \nnot \phi)) \leftrightarrow  \nnot (\phi \land \psi) $.
\item \qquad \label{Itm:HOL8} $\nnot (\phi \to \psi) \leftrightarrow (\phi \to \nnot \psi)$
\end{enumerate}

\sloppy
(The `H' in (HOL\ref{Itm:HOL1}) etc.~refers to `Hilbert'.)
We note that the only non-classically valid scheme is (HOL\ref{Itm:HOL8}), while
(HOL\ref{Itm:HOL1})--(HOL\ref{Itm:HOL6})
constitute, 
with 
\emph{modus ponens}, an axiomatization of the conjunction-implication fragment of classical logic. 
This entails that every  classical tautology in this language is derivable in HOL. 
{Also observe that, since \emph{modus ponens} is the only rule of inference, 
 axioms~(HOL\ref{Itm:HOL1}) and (HOL\ref{Itm:HOL2}) give us that $\to$ satisfies the DDT (see Definition~\ref{def:disj-imp} (2)).} 

{For a quasi-equational presentation of $\mathbb{OL}$, thus, we may employ the following
quasi-equations
(cf.~\cite[Prop.~3.44]{F16}):}
\begin{enumerate}
    \item 
    \qquad $\alpha \approx \abs{\alpha}$ for each axiom $\alpha$ in (HOL~\ref{Itm:HOL1})--(HOL~\ref{Itm:HOL9}),
    \item \qquad if $\alpha \approx \abs{\alpha}$ and $\alpha \to \beta \approx \abs{\alpha \to \beta}$,
    then $\beta \approx \abs{\beta}$,
    \item \qquad if $\alpha \imp \beta \approx \abs{\alpha \imp \beta} $ and
    $\beta \imp \alpha \approx \abs{\beta \imp \alpha} $,
    then $\alpha \approx \beta$.
\end{enumerate}


\section{Future work}
\label{sec:fut}
{Having 
axiomatized
the logic of ordinary discourse $\CooOL$ and 
investigated the logico-algebraic features of
its structural counterpart ($\SOL$),
we believe 
to have contributed to the advancement of
the study of
connexive (multiple- and single-conclusion) logics. 
}
We view the present study as yet another vindication of the usefulness of multiple-conclusion
calculi in the study of finite-valued logics. 
Beyond the results presented here,
we speculate that the following directions may prove fruitful in future research.


    1. Due to space limitations, the algebraic aspects of $\SOL$ have been touched only sketchily in the present paper. We reserve a more comprehensive study -- including a proof of algebraizability of the Hilbert calculus introduced in Section~\ref{sec:alg}, a more perspicuous  presentation of the variety $\mathbb{OL}$, etc. -- to a future publication \cite{Riv24}.

    {2.
    The papers~\cite{egre20211,egre20212} by P.~Egr\'e \emph{et al}.~contain an extensive discussion 
    of    three-valued logics that model conditionals in natural language. 
    Among other  systems, the authors consider
 a variant of $\SOL$ (denoted CC/TT) that employs the implication $\to$ together with the
connectives $\sqcap, \sqcup$ instead of the primitive conjunction and disjunction of $\SOL$.
Present space limitations do not allow us to compare in detail our approach with that of Egr\'e \emph{et al}., so
this issue too will have to be left for future research. However, we may anticipate that our algebraic
analysis of $\SOL$ throws some light on the observations of~\cite[Sec.~4]{egre20212}, in particular
the fact that a translation $\rho \colon \phi \approx \psi \mapsto 
\{ \phi \to \psi, \psi \to \phi
\}$ does not guarantee algebraizability,  either  for $\SOL$ or for CC/TT 
(cf.~\cite[Lemma~4.18]{egre20212}). On the other hand, it is easy to see that the translation considered in Section~\ref{sec:alg}, namely
$\rho \colon \phi \approx \psi \mapsto 
\{ \phi \to \psi, \psi \to \phi, \nnot \phi \to \nnot \psi, \nnot \psi \to \nnot \phi 
\}$, guarantees algebraization for both logics, thereby settling the problem
left open in~\cite{egre20212}.
    } 
    
    3. As mentioned earlier, $\SOL$ is definitionally equivalent to an expansion of 
Da Costa and D'Ottaviano's 
logic $\mathcal{J}3$, which is in turn 
an axiomatic extension of paraconsistent 
Nelson logic. This suggests that $\mathbb{OL}$ may be viewed as a subvariety of N4-lattices and, as such,
its members may be given a twist-structure representation. Developing such a study may not only provide
further insight into $\mathbb{OL}$, but also clarify the relationship between $\SOL$
and other related non-classical  systems, 
such as 
the connexive logic C~\cite{WanConnexive05}.
    
    4. It might be interesting to develop a study similar to the present one for other logics defined from the algebra $\OL_3$
    with different sets of designated values, e.g.~$\{ \tv \}$ or $\{ \uv \}$, or the order-preserving logics associated to the orderings  naturally arising on $\OL_3$
    (cf.~Proposition~\ref{prop:ol3asn4}). An obviously different but related question is whether any of these
    systems admit an interpretation in line with Cooper's original proposal of formalizing
    reasoning in ordinary discourse. 

    
    5. An algebraic study of  the 
    (term-definable) fragments of $\SOL$ 
    axiomatized in Section~\ref{sec:prelims} also appears to be promising. Some of these --  such as the $\{\to, \nnot\}$-fragment and the $\{\land, \nnot\}$-fragment -- 
    are algebraizable, suggesting that they may be easily treatable with algebraic methods. The $\{ \imp \}$-fragment,  
    not considered here, is also easily seen to be algebraizable, and may be axiomatized by 
    the methods employed in the present study. 
    By~\cite{ss1978}, 
    every multiple-conclusion finite-valued logic is axiomatized by a finite multiple-conclusion calculus; 
         given that $\lordimpneg$ was defined using only $\Rightarrow$, a single-conclusion calculus for this fragment
        exists by Theorem~\ref{the:translate-mc-sc}.

\section*{Acknowledgements}
Vitor Greati acknowledges support from the FWF project P33548.
Sérgio Marcelino acknowledges  funding by FCT/MCTES through national funds and when applicable  co-funding by EU under the project UIDB/50008/2020. 
Umberto Rivieccio acknowledges support from the 2023-PUNED-0052
grant ``Investigadores tempranos UNED-SANTANDER''
and  from the I+D+i research project PID2022-142378NB-I00 ``PHIDELO'', funded by the Ministry of Science and Innovation of Spain.


\bibliographystyle{abbrvnat}  
\bibliography{main}

\end{document}